\documentclass[reqno,a4paper, 11pt]{amsart}

\usepackage[a4paper=true,pdfpagelabels]{hyperref}
\usepackage{graphicx}

\usepackage[ansinew]{inputenc}
\usepackage{amsfonts,epsfig}
\usepackage{latexsym}
\usepackage{amsmath}
\usepackage{amssymb}
\usepackage{mathabx}
\usepackage{comment}

\newtheorem{theorem}{Theorem}

\newtheorem{corollary}[theorem]{Corollary}
\newtheorem{proposition}[theorem]{Proposition}

\newtheorem{lettertheorem}{Theorem}
\newtheorem{letterlemma}[lettertheorem]{Lemma}

\theoremstyle{definition}

\newtheorem{example}[theorem]{Example}

\theoremstyle{remark}

\numberwithin{equation}{section}

\newcommand{\D}{\mathbb{D}}
\newcommand{\DD}{\widehat{\mathcal{D}}}
\newcommand{\DDD}{\widecheck{\mathcal{D}}}

\newcommand{\N}{\mathbb{N}}

\newcommand{\R}{\mathbb{R}}
\newcommand{\HO}{\mathcal{H}}

\newcommand{\C}{\mathbb{C}}

\newcommand{\e}{\varepsilon}

\newcommand{\RR}{\mathcal{D}}

\newcommand{\T}{\mathbb{T}}

\newcommand{\whw}{\widehat{\nu}}

\def\a{\alpha}       \def\b{\beta}        \def\g{\gamma}
\def\d{\delta}           \def\e{\varepsilon}
     \def\om{\omega}      
\def\s{\sigma}       \def\t{\theta}       
         \def\r{\rho}

\begin{document}

\title[]{Besov spaces induced by doubling weights}


\thanks{This research was supported in part by Academy of Finland project no.~286877 and Finnish Cultural Foundation.}



\begin{abstract}
Let $1\le p<\infty$, $0<q<\infty$ and $\nu$ be a two-sided doubling weight satisfying
    $$\sup_{0\le r<1}\frac{(1-r)^q}{\int_r^1\nu(t)\,dt}\int_0^r\frac{\nu(s)}{(1-s)^q}\,ds<\infty.$$
The weighted Besov space $\mathcal{B}_{\nu}^{p,q}$ consists of those $f\in H^p$ such that
    $$\int_0^1 \left(\int_{0}^{2\pi} |f'(re^{i\theta})|^p\,d\theta\right)^{q/p}\nu(r)\,dr<\infty.$$
Our main result gives a characterization for $f\in \mathcal{B}_{\nu}^{p,q}$ depending only on $|f|$, $p$, $q$ and $\nu$.

As a consequence of the main result and inner-outer factorization, we obtain several interesting by-products.
For instance, we show the following modification of a classical factorization by F.~and~R.~Nevanlinna:
If $f\in \mathcal{B}_{\nu}^{p,q}$, then there exist $f_1,f_2\in \mathcal{B}_{\nu}^{p,q} \cap H^\infty$ such that $f=f_1/f_2$.
Moreover, we give a sufficient and necessary condition guaranteeing that the product of
$f\in H^p$ and an inner function belongs to $\mathcal{B}_{\nu}^{p,q}$.
Applying this result, we make some observations on zero sets of $\mathcal{B}_{\nu}^{p,p}$.
\end{abstract}

\author{Atte Reijonen}
\address{University of Eastern Finland, P.O.Box 111, 80101 Joensuu, Finland}
\email{atte.reijonen@uef.fi}

\maketitle



\section{Introduction and characterizations}

Let $\D$ be the open unit disc of the complex plane $\C$ and $\T$ the boundary of $\D$.
The set of all analytic functions in $\D$ is denoted by $\HO(\D)$.
For $0<p<\infty$, the Hardy space $H^p$ consists of those $f\in \HO(\D)$ such that
    $$\|f\|_{H^p} =\sup_{0\le r<1} M_p(r,f)<\infty,$$
where
    \begin{equation*}
    \begin{split}
    M_p(r,f)=\left(\frac{1}{2\pi}\int_{0}^{2\pi} |f(re^{i\t})|^p\,d\t\right)^{1/p}.
    \end{split}
    \end{equation*}
The Hardy space $H^\infty$ is the set of all bounded functions in $\HO(\D)$.
Moreover, we recall that a measurable function $f$ on $\T$ belongs to $L^p(\T)$ for some $p\in (0,\infty)$ if
    $$\|f\|_{L^p}^p = \frac{1}{2\pi}\int_0^{2\pi} |f(e^{i\t})|^p\,d\t<\infty.$$
Alternatively, the Hardy space $H^p$ for $0<p<\infty$ can be characterized as follows: $f\in H^p$ if and only if $f\in \HO(\D)$,
non-tangential limit $f(e^{i\t})$ exists almost everywhere on $\T$ and $f(e^{i\t})\in L^p(\T)$.
In particular, $\|f\|_{H^p}=\|f\|_{L^p}$ for $0<p<\infty$ and $f\in H^p$.
This is due to Hardy's convexity and the mean convergence theorems.
These results and much more can be found in classic book \cite{Duren1970} by P.~Duren.

A function $\nu: \D \rightarrow [0,\infty)$ is called a (radial) weight if it is integrable over $\D$ and $\nu(z)=\nu(|z|)$ for all $z\in \D$.
For $0<p,q<\infty$ and a weight $\nu$, the weighted mixed norm space $A^{p,q}_\nu$ consists of those $f\in \HO(\D)$ such that
    $$
    \|f\|_{A^{p,q}_\nu}^q=\int_0^1 M_p^q(r,f)\, \nu(r)\,dr<\infty.
    $$
If $\nu(z)=(1-|z|)^\a$ for $-1<\a<\infty$, then the notation $A_\a^{p,q}$ is used for $A^{p,q}_\nu$.
In this note, we study class $\RR$ of so-called two-sided doubling weights, which originates from the work of J.~A.~ Pel\'aez and J.~R\"atty\"a \cite{PelRat2015, TwoweightII}.
For the definition of $\RR$ we have to define two wider classes.
For a weight $\nu$, set
    \begin{equation*}
    \widehat{\nu}(z)=\widehat{\nu}{(|z|)}=\int_{|z|}^1\nu(s)\,ds,\quad z\in \D.
    \end{equation*}
If a weight $\nu$ satisfies the condition $\widehat{\nu}(r)\le C\widehat{\nu}(\frac{1+r}{2})$ for all $0\le r<1$ and some $C=C(\nu)>0$, then we write $\nu \in \DD$.
Correspondingly, $\nu \in \DDD$ if there exist $K=K(\nu)>1$ and $C=C(\nu)>1$ such that
    \begin{equation*}\label{Eq:definition-DDD}
    \begin{split}
    \widehat{\nu}(r)\ge C \widehat{\nu}\left(1-\frac{1-r}{K}\right), \quad 0\le r<1.
    \end{split}
    \end{equation*}
Class $\RR$ is the intersection of $\DD$ and $\DDD$.
In addition, we define the following subclass of $\DD$: $\nu \in \DD_p$ for some $p\in (0,\infty)$ if the condition
    \begin{equation} \label{Eq:DDp}
    \DD_p(\nu)=\sup_{0\le r<1}\frac{(1-r)^p}{\widehat{\nu}(r)}\int_0^r\frac{\nu(s)}{(1-s)^p}\,ds<\infty
    \end{equation}
is satisfied.
As a concrete example, we mention that $\nu_1(z)=(1-|z|)^\a$ and $\nu_2(z)=(1-|z|)^\a\left(\log\frac{e}{1-|z|}\right)^\b$
for any $\b\in \R$ belong to $\RR \cap \DD_p$ if and only if $-1<\a<p-1$.
Additional information about weights can be found in \cite{Pelaez2014, PelRat2015, TwoweightII}.
Some basic properties are recalled also in Section~\ref{Sec:2}.

Define the weighted Besov space $\mathcal{B}_{\nu}^{p,q}$ by $\mathcal{B}_{\nu}^{p,q}=\{f:f'\in A_\nu^{p,q}\}\cap H^p$.
For $-1<\a<\infty$ and $\nu(z)=(1-|z|)^\a$, the notation $\mathcal{B}_{\a}^{p,q}$ is used for $\mathcal{B}_{\nu}^{p,q}$.
The space $\mathcal{B}_{\nu}^{p,q}$ is the main research objective of this note.
Hence it is worth pointing out that the definition is rational,
which means that $H^p$ is not a subset of $\{f:f'\in A_\nu^{p,q}\}$ in general, or conversely.
The family of Blaschke products offers examples for the case where
$f\in H^\infty$ and $f'\notin A_\nu^{p,q}$; see for instance \cite{RS2018}.
Moreover, it would be natural that certain lacunary series $g$ lie out of $H^p$, while $g'\in A_\nu^{p,q}$.
Arguments for this kind of examples can be found in M.~Pavlovi\'c's book \cite{Pavlovic2014}, which contains
numerous important observations on the topic of this note.
The existence of both examples, of course, depends on $p$, $q$ and $\nu$.
In other words, under certain hypotheses for $p$, $q$ and $\nu$,
an inclusion relation between $\{f:f'\in A_\nu^{p,q}\}$ and $H^p$ might be valid.
However, this is not the case in general.

For $0<p<\infty$ and $f\in L^p(\T)$, the $L^p$ modulus of continuity $\om_p(t,f)$ is defined by
    $$\om_p(t,f)=\sup_{0<h<t} \left(\int_0^{2\pi} |f(e^{i(\t+h)})-f(e^{i\t})|^p \,d\t\right)^{1/p}, \quad 0<t\le 2\pi.$$
We interpret $\om_p(t, f)=\om_p(2\pi, f)$ for $t>2\pi$. It is a well-known fact that, for $0<p,q<\infty$, $-1<\a<q-1$ and $f\in H^p$,
the derivative of $f$ belongs to $A_\a^{p,q}$ if and only if
    $$\int_0^\infty \frac{\om_p(t,f)^q}{t^{q-\a}}\,dt<\infty.$$
This result originates from E.~M.~Stein's book \cite[Chapter V, Section 5]{Stein1970},
and the complete version is a consequence of \cite[Theorems~2.1~and~5.1]{Pavlovic1992} or \cite[Theorem~1.2]{JP2013} by M.~Pavlovi\'c and M.~Jevti\'c.
Our first theorem is a partial generalization of the result.
Its proof uses some ideas from \cite{Duren1970, PR2016, R2017b}.

\begin{theorem}\label{Thm1}
Let $1\le p<\infty$, $0<q<\infty$ and $\nu \in \RR$. Then $\nu \in \DD_q$ if and only if
there exists a constant $C=C(p,q,\nu)>0$ such that
    \begin{equation} \label{Eq:SP}
    \int_{1/2}^1 \om_p(1-r, f)^q \frac{\nu(r)}{(1-r)^q}\,dr \le C \|f'\|_{A_\nu^{p,q}}^q
    \end{equation}
for all $f\in H^p$.
\end{theorem}

Note that \eqref{Eq:SP} is valid also if $0<p,q<\infty$, $\nu$ is a weight and there exists $\b=\b(q,\nu)<q-1$ such that $\nu(r)/(1-r)^{\b}$
is increasing for $0\le r<1$.
This is due to \cite[Theorem~5.1]{Pavlovic1992} and its proof.
Even though the result is valid also for $0<p<1$, Theorem~\ref{Thm1} is more useful for our purposes.
In particular, it is worth underlining that the hypothesis $1\le p<\infty$ in Theorems~\ref{Thm2}~and~\ref{Thm3} below is natural.

Theorem~\ref{Thm2} gives a practical estimate for $\|f'\|_{A_\nu^{p,q}}^q+\|f\|_{H^p}^q$ when $f\in H^p$.
As a by-product of its argument, we deduce that also the converse inequality of \eqref{Eq:SP} holds
if the norm $\|f\|_{H^p}^q$ is added into the right-hand side.
Setting
    $$d\mu_{z}(\t)=\frac{1-|z|^2}{|e^{i\t}-z|^2}\frac{d\t}{2\pi}, \quad z\in \D, \quad 0\le \t<2\pi,$$
Theorem~\ref{Thm2} reads as follows.

\begin{theorem}\label{Thm2}
Let $1\le p<\infty$, $0<q<\infty$ and $\nu \in \RR \cap \DD_q$.
Then there exist positive constants $C_1$ and $C_2$ depending only on $p$, $q$ and $\nu$ such that
    \begin{equation}\label{Eq:SP3}
    \begin{split}
    \|f'\|_{A_\nu^{p,q}}^q& \le \int_0^1 \left(\int_0^{2\pi} \left(\int_0^{2\pi} |f(e^{i\t})-f(re^{it})| d\mu_{re^{it}}(\t) \right)^p dt\right)^{q/p} \frac{\nu(r)}{(1-r)^q}\,dr \\
    &\le C_1 \left(\int_0^1 \om_p(1-s, f)^q \frac{\nu(s)}{(1-s)^{q}}\,ds + \|f\|_{H^p}^q \right) \\
    &\le C_2 \left (\|f'\|_{A_\nu^{p,q}}^q+\|f\|_{H^p}^q \right )
    \end{split}
    \end{equation}
for all $f\in H^p$.
\end{theorem}

By studying the classical weight $\nu(z)=(1-|z|)^\a$, where $-1<\a<q-1$,
we obtain K.~M.~Dyakonov's \cite[Proposition~2.2(a)]{Dyakonov1998} as a direct consequence of Theorem~\ref{Thm2}.
Hence it does not come as a surprise that the proofs of \cite[Theorem~2.1]{Dyakonov1998} and Theorem~\ref{Thm2} have some similarities.
Nonetheless, it is worth mentioning that the presence of general weights complicates the argument;
and consequently, our proof is quite technical. Note also that Theorem~\ref{Thm1} plays an essential role in the proof.

Our main result below gives a characterization for functions $f$ in $\mathcal{B}_{\nu}^{p,q}$ depending only $|f|$, $p$, $q$ and $\nu$.
This result improves B.~B\o e's \cite[Theorem~1.1]{Boe2003}, which concentrates only on the case where $1\le p,q<\infty$, $-1<\a<q-1$ and $\nu(z)=(1-|z|)^\a$.
It also generalizes the essential contents of \cite[Proposition~2.4]{Aleman1992} and \cite[Proposition~2.2(b)]{Dyakonov1998} made by A.~Aleman and K.~M.~Dyakonov, respectively.

\begin{theorem}\label{Thm3}
Let $1\le p<\infty$, $0<q<\infty$ and $\nu \in \RR \cap \DD_q$.
Then there exist positive constants $C_1$ and $C_2$ depending only on $p$, $q$ and $\nu$ such that
    \begin{equation} \label{Eq:SP4}
    \begin{split}
    &\|f'\|_{A_\nu^{p,q}}^q \le C_1\left(F_1(f)+F_2(f)\right) \le C_2\left(\|f'\|_{A_\nu^{p,q}}^q + \|f\|_{H^p}^q\right), \quad f\in H^p,
    \end{split}
    \end{equation}
where
    \begin{equation*}
    \begin{split}
    F_1(f)=\int_0^1 \left(\int_0^{2\pi} \left(\int_0^{2\pi} |f(e^{i\t})| d\mu_{re^{it}}(\t)-|f(re^{it})|\right)^p dt\right)^{q/p} \frac{\nu(r)}{(1-r)^q}\,dr
    \end{split}
    \end{equation*}
and
    \begin{equation*}
    \begin{split}
    F_2(f)=\int_0^1 \left(\int_0^{2\pi} \left(\int_0^{2\pi} \left||f(e^{i\t})|-\int_0^{2\pi} |f(e^{is})| d\mu_{re^{it}}(s)\right| d\mu_{re^{it}}(\t) \right)^p dt\right)^{q/p} \frac{\nu(r)}{(1-r)^q}\,dr.
    \end{split}
    \end{equation*}
\end{theorem}

Before we talk about the argument of Theorem~\ref{Thm3}, recall the inner-outer factorization.
An inner function is a member of $H^\infty$ having unimodular radial limits almost everywhere on $\T$.
For $0<p\le \infty$, an outer function for $H^p$ takes the form
    $$O_\phi(z)=\exp\left(\frac{1}{2\pi}\int_0^{2\pi}  \frac{e^{i\t}+z}{e^{i\t}-z}\log \phi(e^{i\t})\,d\t\right), \quad z\in \D,$$
where $\phi$ is a non-negative function in $L^p(\T)$ and $\log \phi \in L^1(\T)$.
The inner-outer factorization asserts that $f\in H^p$ can be represented as the product of an inner and outer function; see for instance \cite[Theorem~2.8]{Duren1970}.
It is worth noting that the factorization is unique, and
    \begin{equation}\label{Eq:inner-outer-boundary}
    \begin{split}
    |f(\xi)|=|O_\phi(\xi)|=\phi(\xi)
    \end{split}
    \end{equation}
for almost every $\xi \in \T$ if $O_\phi$ is the outer function from the factorization of $f$.
Equation \eqref{Eq:inner-outer-boundary} is due to the definition of inner functions,
Poisson integral formula, harmonicity of $\log |O_{\phi}(z)|$ and fact that
    $$|O_\phi(z)|=\exp\left(\int_0^{2\pi} \log \phi(e^{i\t})\,d\mu_z(\t)\right), \quad z\in \D.$$

The last inequality in \eqref{Eq:SP4} can be proved by applying Theorem~\ref{Thm2}.
In the argument of the first inequality, the inner-outer factorization, Schwarz-Pick lemma
and an upper estimate for $|O_\phi'|$ from \cite{Boe2003} are the main tools.
It is worth underlining that this B\o e's idea to make an upper estimate for $|f'|$
by using the factorization seems to be quite effective.
Another way to prove results like Theorem~\ref{Thm3} is to use a modification of Theorem~\ref{Thm2} together with the well-known equation
    \begin{equation*}
    \begin{split}
    \int_0^{2\pi} |f(e^{i\t})-f(z)|^2 d\mu_z(\t)=\int_0^{2\pi} |f(e^{i\t})|^2 d\mu_z(\t)-|f(z)|^2, \quad z\in\D;
    \end{split}
    \end{equation*}
but this Dyakonov's method has the obvious defect that it works only when $f\in H^2$.
The advantage of this method in the case where $2\le p<\infty$, $0<q<\infty$, $q/2-1<\a<q-1$ and $\nu(z)=(1-|z|)^\a$ is that $F_1(f)+F_2(f)$ in Theorem~\ref{Thm3}
can be replaced by
    \begin{equation*}
    \begin{split}
    \int_0^1 \left(\int_0^{2\pi} \left(\int_0^{2\pi} |f(e^{i\t})|^2 d\mu_{re^{it}}(\t)-|f(re^{it})|^2\right)^{p/2} dt\right)^{q/p} \frac{\nu(r)}{(1-r)^q}\,dr;
    \end{split}
    \end{equation*}
see \cite[Proposition~2.2(b)]{Dyakonov1998}.
It is an open problem to prove a corresponding estimate for general weights.

Next we give an example which shows that the hypothesis $\nu \in \RR \cap \DD_q$
in Theorem~\ref{Thm3} for $p\ge 2$ is sharp in a certain sense.
Note that the example is a modification of \cite[Example~8]{R2017b}.
Before the statement we fix some notation. Write
$f\lesssim g$ if there exists a constant $C>0$ such that $f\le Cg$, while $f\gtrsim g$ is understood analogously.
If $f\lesssim g$ and $f\gtrsim g$, then we write $f\asymp g$.

\begin{example}\label{Example1}
Let $2\le p<\infty$, $q=p$, $\nu(z)=(1-|z|)^{p-1}$ and $dA(z)$ be the two-dimensional Lebesgue measure $dxdy$.
Let $f$ be an inner function such that
    $$\int_{\{z\in \D:|f(z)|<\e\}}\frac{dA(z)}{1-|z|}=\infty$$
for some $\e\in (0,1)$. The existence of such $f$ is guaranteed by \cite[Theorem~5]{Borichev2013}.
Then
    \begin{equation*}
    \begin{split}
    F_1(f)+F_2(f)=F_1(f)\asymp \int_{\D} (1-|f(z)|)^p (1-|z|)^{-1}\,dA(z)=\infty,
    \end{split}
    \end{equation*}
while
    \begin{equation*}
    \begin{split}
    \|f'\|_{A_\nu^{p,q}}^q + \|f\|_{H^p}^q=\|f'\|_{A_{p-1}^{p,p}}^p+1<\infty
    \end{split}
    \end{equation*}
by the well-known inclusion
    $$H^p \subset \{g:g'\in A_{p-1}^{p,p}\}, \quad 2\le p<\infty,$$
which originates from \cite{LP1936}.
\end{example}

We close the section by explaining how the remainder of this note is organized.
Auxiliary results on weights are recalled in the next section.
The utility of Theorem~\ref{Thm3} is demonstrated in Sections~\ref{Sec:3}~and~\ref{Sec:4}.
More precisely, in Section~\ref{Sec:3}, we prove the factorization which states that, for any $f\in \mathcal{B}_{\nu}^{p,q}$,
there exist $f_1,f_2\in \mathcal{B}_{\nu}^{p,q} \cap H^\infty$ such that $f=f_1/f_2$.
Section~\ref{Sec:4} begins with a result giving a sufficient and necessary condition guaranteeing that the product of
$f\in H^p$ and an inner function belongs to $\mathcal{B}_{\nu}^{p,q}$.
As a consequence of this theorem, we obtain some results on zero sets of $\mathcal{B}_{\nu}^{p,p}$.
Sections~\ref{Sec:5},~\ref{Sec:6}~and~\ref{Sec:7} consist of the proofs of Theorems~\ref{Thm1},~\ref{Thm2}~and~\ref{Thm3}, respectively.

\section{Auxiliary results on weights}\label{Sec:2}

In this section, we recall some basic properties of weights in $\DD$ and $\DDD$.
These properties are needed in next sections.
Another reason for these results is to help the reader to understand the nature of weights in $\RR$.
We begin with a result which is essentially \cite[Lemma~3]{PelRat2015}; see also \cite[Lemma~2.1]{Pelaez2014}.

\begin{letterlemma}\label{Lemma:DD}
Let $\nu$ be a weight. Then the following statements are equivalent:
\begin{itemize}
    \item[(i)] $\nu\in \DD$.
    \item[(ii)] There exist $C=C(\nu)>0$ and $\b=\b(\nu)>0$ such that
        \begin{equation*}
        \whw(r)\le C\left(\frac{1-r}{1-s}\right)^\b \whw(s), \quad 0\le r\le s<1.
        \end{equation*}
    \item[(iii)] There exist $C=C(\nu)>0$ and $\g=\g(\nu)>0$ such that
        \begin{equation*}
        \int_0^r \left(\frac{1-r}{1-s}\right)^\g\nu(s)\,ds\le C\whw(r), \quad 0\le r<1.
        \end{equation*}
    \item[(iv)] The estimate
        $$\int_0^1 s^x\nu(s)\,ds \asymp \whw\left(1-\frac{1}{x}\right), \quad 1\le x<\infty,$$
        is satisfied.
\end{itemize}
\end{letterlemma}

For the point view of our main results Lemma~\ref{Lemma:DD}(iii) is interesting because it states that $\nu \in\DD$ if and only if $\nu \in \DD_{p}$ for some $p>0$.
This means that $\DD=\bigcup_{p>0} \DD_p$. Nevertheless, Lemma~\ref{Lemma:DD}(ii) gives maybe the most interesting description for $\DD$.
Together with its $\DDD$ counterpart below it offers a very practical characterization for weights in $\RR$.
Essentially this characterization says that $\whw$ is normal in the sense of A.~L.~Shields and D.~L.~Williams \cite{SW1971}.

\begin{letterlemma}\label{Lemma:DDD}
Let $\nu$ be a weight. Then $\nu \in \DDD$ if and only if there exist $C=C(\nu)>0$ and $\a=\a(\nu)>0$ such that
        \begin{equation*}
        \whw(s)\le C\left(\frac{1-s}{1-r}\right)^\a \whw(r), \quad 0\le r\le s<1.
        \end{equation*}
\end{letterlemma}

Lemma~\ref{Lemma:DDD} originates from \cite{TwoweightII}, and it can be proved in a corresponding manner as Lemma~\ref{Lemma:DD}(ii).
See in particular the proof of \cite[Lemma~2.1]{Pelaez2014}.

By the definition of class $\DD_p$, it is clear that $\DD_p \subset \DD_{p+\e}$ for any $\e>0$.
Next we state \cite[Lemma~3]{PR2016}, which shows that also the converse inclusion is true for sufficiently small $\e=\e(\nu, p)>0$.
The proof of this result is based on integration by parts.
Note that $\DD_p(\nu)$ in the statement is defined by \eqref{Eq:DDp}.

\begin{letterlemma}\label{Lemma:self-imp}
If $0<p<\infty$ and $\nu\in \DD_p$, then $\nu\in \DD_{p-\e}$ for any $\e \in \left(0,\frac{p}{\DD_p(\nu)+1}\right)$.
\end{letterlemma}

The last result of this section is \cite[Lemma~5]{R2017b}, which shows that $\nu\in \RR$ in the norm $\|f\|_{A_\nu^{p,q}}$ can be replaced by $\whw(z)/(1-|z|)$ without losing any essential information.

\begin{letterlemma}\label{Lemma:asymp}
Let $0<p,q<\infty$ and $\nu$ be a weight.
\begin{itemize}
    \item[(i)] If $\nu \in \DDD$, then there exists $C=C(\nu)>0$ such that
        $$\|f\|_{A_\nu^{p,q}}^q\ge C \int_0^1 M_p^q(r,f)\,\frac{\whw(r)}{1-r}\,dr, \quad f\in \HO(\D).$$
    \item[(ii)] If $\nu \in \DD$, then there exists $C=C(\nu)>0$ such that
        $$\|f\|_{A_\nu^{p,q}}^q\le C \int_0^1 M_p^q(r,f)\,\frac{\whw(r)}{1-r}\,dr, \quad f\in \HO(\D).$$
\end{itemize}
\end{letterlemma}

For $\nu \in \RR$, Lemmas~\ref{Lemma:DD}(ii)~and~\ref{Lemma:DDD} yield
    \begin{equation}\label{Eq:weight-scal}
    \begin{split}
    \whw(r)\asymp \int_r^1\frac{\whw(s)}{1-s}\,ds, \quad 0\le r<1.
    \end{split}
    \end{equation}
In \cite{R2017b},
Lemma~\ref{Lemma:asymp} is proved by applying this fact together with partial integrations.
An alternative way to prove results like Lemma~\ref{Lemma:asymp} is to split the integral with respect to $dr$
to infinitely many parts by using a dyadic partition, and then apply \eqref{Eq:weight-scal} together with the monotonicity of $M_p^q(r,f)$.
An advantage of the last method is that $M_p^q(r,f)$ can be easily replaced by a certain monotonic function $g(r)$.
This observation will be utilized several times in the argument of Theorem~\ref{Thm2}.

\section{Quotient factorization}\label{Sec:3}

Recall that if $f\in H^p$ for some $p>0$, then there exist $f_1,f_2\in H^\infty$ such that $f=f_1/f_2$.
This is an important consequence of classical factorization \cite[Theorem~2.1]{Duren1970} by F.~and~R.~Nevanlinna.
The main purpose of this section is to give the following $\mathcal{B}_{\nu}^{p,q}$ counterpart for the above-mentioned result.

\begin{theorem}\label{Thm4}
Let $1\le p<\infty$, $0<q<\infty$ and $\nu \in \RR \cap \DD_q$. If $f\in \mathcal{B}_{\nu}^{p,q}$, then
there exist $f_1,f_2\in \mathcal{B}_{\nu}^{p,q} \cap H^\infty$ such that $f=f_1/f_2$ and $f_2$ is an outer function.
\end{theorem}

It is worth mentioning that \cite[Theorem~9.19]{Pavlovic2014} is a similar type of result as Theorem~\ref{Thm4} with a different hypothesis for $\nu$.
Moreover, we note that Theorem~\ref{Thm4} generalizes \cite[Corollary~2.7]{Aleman1992}, \cite[Theorem~3.4]{Boe2003} and \cite[Corollary~3.4]{Dyakonov1998}.
For its argument we need an extension of \cite[Theorem~3.3]{Boe2003}.
Note that a part of the next pages is really inspired by \cite{Boe2003}.

\begin{proposition}\label{Prop-inner-outer}
Let $1\le p<\infty$, $0<q<\infty$, $\nu \in \RR \cap \DD_q$ and $f\in H^p$
be the product of an inner function $I$ and an outer function $O_\phi$.
Then there exists a constant $C=C(p,q,\nu)>0$ such that
    \begin{equation*}
    \begin{split}
    \|O_{\max\{\phi,1\}}'\|_{A_\nu^{p,q}}^q + \|(IO_{\min\{\phi,1\}})'\|_{A_\nu^{p,q}}^q + \|O_{\max\{\phi,1\}}\|_{H^p}^q
    \le C \left(\|f'\|_{A_\nu^{p,q}}^q + \|f\|_{H^p}^q + 1\right).
    \end{split}
    \end{equation*}
\end{proposition}

Before the proof of Proposition~\ref{Prop-inner-outer}, we note that the quantities $F_1(f)$ and $F_2(f)$
in Theorem~\ref{Thm3} are used repeatedly in the future.

\begin{proof}
Let us begin by noting that $|O_{\phi}(e^{i\t})|=\phi(e^{i\t})$, $|O_{\max\{\phi,1\}}(e^{i\t})|=\max\{\phi(e^{i\t}),1\}$ and
     \begin{equation*}
    \begin{split}
    \max\{\phi(e^{i\t}),1\}-\phi(e^{i\t})&=\frac{\max\{\phi(e^{i\t}),1\}-\phi(e^{i\t})}{\max\{\phi(e^{i\t}),1\}} \\
    &\le |O_{\max\{\phi,1\}}(z)|\left(1-\frac{\phi(e^{i\t})}{\max\{\phi(e^{i\t}),1\}}\right)
    \end{split}
    \end{equation*}
for all $z\in \D$ and almost every $\t\in [0,2\pi)$. Using these facts together with
Jensen's inequality \cite[Chapter~I,~Lemma~6.1]{Garnett1981} and the definition of outer functions,
we obtain
    \begin{equation*}
    \begin{split}
    &\int_0^{2\pi} |O_{\max\{\phi,1\}}(e^{i\t})|\,d\mu_z(\t) - \int_0^{2\pi}|O_\phi(e^{i\t})|\,d\mu_z(\t) \\
    &\quad \le |O_{\max\{\phi,1\}}(z)|\left(1-\int_0^{2\pi} \frac{\phi(e^{i\t})}{\max\{\phi(e^{i\t}),1\}}\,d\mu_z(\t)\right)\\
    &\quad = |O_{\max\{\phi,1\}}(z)|\left(1-\int_0^{2\pi} \exp\left(\log \phi(e^{i\t})-\log \max\{\phi(e^{i\t}),1\}\right)\,d\mu_z(\t)\right)\\
    &\quad \le |O_{\max\{\phi,1\}}(z)|\left(1-\frac{|O_{\phi}(z)|}{|O_{\max\{\phi,1\}}(z)|}\right), \quad z\in\D.
    \end{split}
    \end{equation*}
Consequently, the obvious inequality $|f(z)|\le |O_\phi(z)|$ yields
    \begin{equation}\label{Eq:prop-proof-2}
    \begin{split}
    \int_0^{2\pi} |O_{\max\{\phi,1\}}(e^{i\t})|\,d\mu_z(\t) - |O_{\max\{\phi,1\}}(z)|
    \le \int_0^{2\pi} |f(e^{i\t})|d\mu_z(\t)-|f(z)|, \quad z\in \D.
    \end{split}
    \end{equation}
Write $z=re^{it}$. Raising both sides of \eqref{Eq:prop-proof-2} to power $p$, integrating from $0$ to $2\pi$ with respect to $dt$, then raising both sides to power $q/p$ and finally integrating from $0$ to $1$ with respect to $\nu(r) dr/(1-r)^q$, we obtain $F_1(O_{\max\{\phi,1\}})\le F_1(f)$.

Next we show $F_2(O_{\max\{\phi,1\}})\le F_2(f)$. Set
    \begin{equation*}
    \begin{split}
    \Gamma_1=\Gamma_1(z,\phi)=\left\{\t\in [0,2\pi):\int_0^{2\pi} \max\{\phi(e^{is}),1\} d\mu_z(s)\le \phi(e^{i\t})\right\}
    \end{split}
    \end{equation*}
and
    \begin{equation*}
    \begin{split}
    \Gamma_2=\Gamma_2(z,\phi)=\left\{\t\in [0,2\pi):\int_0^{2\pi} \phi(e^{is}) d\mu_z(s)\le \phi(e^{i\t})\right\}, \quad z\in \D.
    \end{split}
    \end{equation*}
Then elementary calculations yield
    \begin{equation}\label{Eq:Gamma12-est}
    \begin{split}
    &\int_0^{2\pi} \left|\max\{\phi(e^{i\t}),1\} -\int_0^{2\pi} \max\{\phi(e^{is}),1\}  d\mu_z(s)\right| d\mu_z(\t) \\
    &\quad =2\int_{\Gamma_1}\left(\max\{\phi(e^{i\t}),1\} -\int_0^{2\pi} \max\{\phi(e^{is}),1\}  d\mu_z(s)\right)\,d\mu_z(\t)\\
    &\quad \quad+ \int_0^{2\pi}\left(\int_0^{2\pi} \max\{\phi(e^{is}),1\}  d\mu_z(s)-\max\{\phi(e^{i\t}),1\} \right)\,d\mu_z(\t) \\
    &\quad= 2\int_{\Gamma_1}\left(\phi(e^{i\t})-\int_0^{2\pi} \max\{\phi(e^{is}),1\} d\mu_z(s)\right)\,d\mu_z(\t) \\
    &\quad\le 2\int_{\Gamma_2}\left(\phi(e^{i\t})-\int_0^{2\pi} \phi(e^{is}) d\mu_z(s)\right)\,d\mu_z(\t) \\
    &\quad =\int_0^{2\pi} \left|\phi(e^{i\t}) -\int_0^{2\pi} \phi(e^{is}) d\mu_z(s)\right| d\mu_z(\t), \quad z\in \D.
    \end{split}
    \end{equation}
Consequently, we obtain $F_2(O_{\max\{\phi,1\}})\le F_2(O_\phi)=F_2(f)$ by doing a similar integral procedure as above.
Now Theorem~\ref{Thm3} together with the inequalities for $F_1(f)$ and $F_2(f)$ gives
    \begin{equation}\label{Eq:prop-proof-3}
    \begin{split}
    \|O_{\max\{\phi,1\}}'\|_{A_\nu^{p,q}}^q + \|O_{\max\{\phi,1\}}\|_{H^p}^q
    \lesssim \|f'\|_{A_\nu^{p,q}}^q + \|f\|_{H^p}^q + 1.
    \end{split}
    \end{equation}

By \eqref{Eq:prop-proof-3} it suffices to show
    \begin{equation}\label{Eq:prop-proof-4}
    \begin{split}
    \|(IO_{\min\{\phi,1\}})'\|_{A_\nu^{p,q}}^q \lesssim \|f'\|_{A_\nu^{p,q}}^q + \|f\|_{H^p}^q.
    \end{split}
    \end{equation}
Since
     \begin{equation*}
    \begin{split}
    \phi(e^{i\t})-\min\{\phi(e^{i\t}),1\}\ge |O_{\min\{\phi,1\}}(z)|\left(\frac{\phi(e^{i\t})}{\min\{\phi(e^{i\t}),1\}}-1\right),
    \end{split}
    \end{equation*}
we obtain
     \begin{equation*}
    \begin{split}
    \int_0^{2\pi} |O_{\phi}(e^{i\t})|d\mu_z(\t)-|O_{\phi}(z)|
    \ge \int_0^{2\pi} |O_{\min\{\phi,1\}}(e^{i\t})|\,d\mu_z(\t) - |O_{\min\{\phi,1\}}(z)|, \quad z\in \D,
    \end{split}
    \end{equation*}
by arguing as above using Jensen's inequality.
It follows that
    \begin{equation*}
    \begin{split}
    &\int_0^{2\pi} |f(e^{i\t})|d\mu_z(\t)-|f(z)|=\left(\int_0^{2\pi} |O_{\phi}(e^{i\t})|d\mu_z(\t)-|O_{\phi}(z)|\right)+|O_{\phi}(z)|(1-|I(z)|) \\
    &\quad \ge \left(\int_0^{2\pi} |O_{\min\{\phi,1\}}(e^{i\t})|d\mu_z(\t)-|O_{\min\{\phi,1\}}(z)|\right)+|O_{\min\{\phi,1\}}(z)|(1-|I(z)|) \\
    &\quad = \int_0^{2\pi} |IO_{\min\{\phi,1\}}(e^{i\t})|d\mu_z(\t)-|IO_{\min\{\phi,1\}}(z)|, \quad z\in \D.
    \end{split}
    \end{equation*}
Hence it is easy to deduce $F_1(IO_{\min\{\phi,1\}})\le F_1(f)$. Since
    $$F_2(IO_{\min\{\phi,1\}})=F_2(O_{\min\{\phi,1\}})\le F_2(O_\phi) =F_2(f)$$
can be shown by using a modification of \eqref{Eq:Gamma12-est}, the desired estimate \eqref{Eq:prop-proof-4} follows from Theorem~\ref{Thm3}.
This completes the proof.
\end{proof}

Now we can easily prove Theorem~\ref{Thm4} by using Proposition~\ref{Prop-inner-outer}.

\medskip

\noindent
\emph{Proof of Theorem~\ref{Thm4}.} By the inner-outer factorization,
there exist an inner function $I$ and an outer function $O_{\phi}$ such that $f=IO_{\phi}$.
Since $O_{\phi}=O_{\min\{\phi,1\}} O_{\max\{\phi,1\}}$, we have $f=f_1/f_2$,
where $f_1=IO_{\min\{\phi,1\}}$ and $f_2=1/O_{\max\{\phi,1\}}$.
Applying Proposition~\ref{Prop-inner-outer} together with the inequalities
    $$|O_{\min\{\phi,1\}}(z)|\le 1 \le |O_{\max\{\phi,1\}}(z)|$$
and
    $$|f_2'(z)|
    \le |O_{\max\{\phi,1\}}(z)|^2|f_2'(z)|
    =|O_{\max\{\phi,1\}}'(z)|, \quad z\in \D,$$
we can check that $f_1$ and $f_2$ belong to $\mathcal{B}_{\nu}^{p,q}\cap H^\infty$.
Moreover, it is obvious that $f_2$ is an outer function.
Hence the proof is complete.  \hfill$\Box$

\section{Product of $f\in H^p$ and an inner function in $\mathcal{B}_{\nu}^{p,q}$}\label{Sec:4}

Theorem~\ref{Thm:fI-SP} below gives a sufficient and necessary condition guaranteeing that the product of
$f\in H^p$ and an inner function belongs to $\mathcal{B}_{\nu}^{p,q}$.
This result generalizes \cite[Corollary~3.2]{Boe2003}, the essential contents of \cite[Corollary~3.1]{Boe2003} and \cite[Theorem~3.2]{Dyakonov1998}.

\begin{theorem}\label{Thm:fI-SP}
Let $1\le p<\infty$, $0<q<\infty$, $\nu \in \RR \cap \DD_q$, $f\in H^p$ and $I$ be an inner function.
Then $fI\in \mathcal{B}_{\nu}^{p,q}$ if and only if $f\in \mathcal{B}_{\nu}^{p,q}$ and
    \begin{equation*}
    \begin{split}
    \int_0^1 \left(\int_0^{2\pi} \left(\frac{|f(re^{it})|(1-|I(re^{it})|)}{1-r}\right)^p\,dt \right)^{q/p}\,\nu(r)\,dr<\infty.
    \end{split}
    \end{equation*}
\end{theorem}

\begin{proof}
We have
    \begin{equation}\label{Eq:fI-SP}
    \begin{split}
    \int_0^{2\pi} |fI(e^{i\t})|d\mu_z(\t)-|fI(z)|=\left(\int_0^{2\pi} |f(e^{i\t})|d\mu_z(\t)-|f(z)|\right)+|f(z)|(1-|I(z)|) \\
    \end{split}
    \end{equation}
for all $z\in \D$. Write $z=re^{it}$. Raising both sides of \eqref{Eq:fI-SP} to power $p$, integrating from $0$ to $2\pi$ with respect to $dt$, then raising both sides to power $q/p$, integrating from $0$ to $1$ with respect to $\nu(r) dr/(1-r)^q$ and finally splitting the right-hand side into two parts by using well-known inequalities, we obtain
    \begin{equation*}
    \begin{split}
    F_1(fI)\asymp F_1(f)+\int_0^1 \left(\int_0^{2\pi} \left(\frac{|f(re^{it})|(1-|I(re^{it})|)}{1-r}\right)^p\,dt\right)^{q/p}\,\nu(r)\,dr.
    \end{split}
    \end{equation*}
Since
    $$F_2(fI)+\|fI\|_{H^p}^q=F_2(f)+\|f\|_{H^p}^q,$$
the assertion follows from Theorem~\ref{Thm3}.
\end{proof}

Recall that a subspace $X$ of $H^p$ satisfies the $F$-property if the hypothesis $fI\in X$, where $f\in H^p$ and $I$ is an inner function,
implies $f\in X$. The $F$-property for $\mathcal{B}_{\nu}^{p,q}$ is a direct consequence of Theorem~\ref{Thm:fI-SP}.
However, it is worth mentioning that if one just aims to
prove the $F$-property for $\mathcal{B}_{\nu}^{p,q}$, our argument is not maybe the
simplest one, taking into account the length of proofs of Theorem~\ref{Thm3} and its auxiliary results. Ideas for an alternative proof can be found,
for instance, in \cite[Section~5.8.3]{Pavlovic2014}.

A sequence $\{z_n\}\subset \D$ is said to be a zero set of $\mathcal{B}_{\nu}^{p,q}$ if there exists $f\in \mathcal{B}_{\nu}^{p,q}$
such that $\{z:f(z)=0\}=\{z_n\}$. Here each zero $z_n$ is repeated according to its multiplicity and function $f$ is not identically zero.
Applying Theorem~\ref{Thm:fI-SP}, we make some observations on zero sets of $\mathcal{B}_{\nu}^{p,p}$.
More precisely, we concentrate on the case where $\{z_n\}$ is separated, which means that there exists $\d=\d(\{z_n\})>0$ such that
$d(z_n,z_k)>\d$ for all $n\neq k$, where
    \begin{equation*}\label{separated}
    d(z,w)=\left|\frac{z-w}{1-\overline{z}w}\right|,\quad z,w\in \D,
    \end{equation*}
is the pseudo-hyperbolic distance between points $z$ and $w$.
Before these results some basic properties of Hardy spaces are recalled.

For $\{z_n\}\subset\D$ satisfying the Blaschke condition $\sum_n (1-|z_n|)<\infty$ and a point $\t \in [0,2\pi)$, the Blaschke product
with zeros $\{z_n\}$ is defined by
    \begin{equation*}\label{Eq:Blaschke}
    B(z)=e^{i\t} \prod_{n}\frac{|z_n|}{z_n}\frac{z_n-z}{1-\overline{z}_nz}, \quad z\in \D.
    \end{equation*}
For $z_n=0$, the interpretation $|z_n|/z_n=-1$ is used.
By factorization \cite[Theorem~2.5]{Duren1970} made by F.~Riesz, we know that any $f\in H^p$ for some fixed $p\in (0,\infty]$ can be represented in the form $f=Bg$, where
$B$ is a Blaschke product and $g\in H^p$ does not vanish in $\D$. More precisely, Beurling factorization \cite[Theorem~2.8]{Duren1970} asserts that $g$ is the product
of an outer function and a singular inner function
    $$
    S(z)=\exp\left(\int_{\T} \frac{z+\xi}{z-\xi}\, d\s(\xi)+i\t \right),\quad z\in\D,
    $$
where $\t \in [0,2\pi)$ is a constant and $\s$ a positive measure on $\T$, singular with respect to the Lebesgue measure.
Consequently, every zero set of $\mathcal{B}_{\nu}^{p,q}$ satisfies the Blaschke condition. With these preparations we are ready to state and prove the following result.

\begin{corollary}\label{Coro:zero}
Let $1\le p<\infty$, $\nu \in \RR \cap \DD_p$, and assume that $\{z_n\}$ is a finite union of separated sequences and zero set of $\mathcal{B}_{\nu}^{p,p}$.
Then there exists an outer function $O_{\phi} \in \mathcal{B}_{\nu}^{p,p}$ such that
    \begin{equation*}
    \begin{split}
    \sum_n |O_{\phi}(z_n)|^p \frac{\whw(z_n)}{(1-|z_n|)^{p-1}}<\infty.
    \end{split}
    \end{equation*}
\end{corollary}

\begin{proof}
Let $\{z_n\}=\bigcup_{j=1}^M \{z_n^j\}$, where $M\in \N$ and each $\{z_n^j\}$ is separated.
Let $B$ be the Blaschke product with zeros $\{z_n\}$, $S$ a singular inner function and $O_{\phi}$ an outer function such that $BSO_{\phi}\in \mathcal{B}_{\nu}^{p,q}$.
By Theorem~\ref{Thm:fI-SP}, we know that $O_{\phi}$ and $BO_{\phi}$ belong to $\mathcal{B}_{\nu}^{p,p}$.
For $w\in \D$ and $0<r<1$, set
    $$\Delta(w,r)=\{z:d(z,w)<r\} \quad \text{and} \quad \Lambda(w,r)=\{z:|w-z|<r(1-|w|)\}.$$
Since each $\{z_n^j\}$ is separated, we find $R_j,\d_j\in (0,1)$
such that, for a fixed $j$, discs $\Lambda(z_n^j, R_j)$ are pairwise disjoint and the inclusion
    $\Delta(z_n^j, \d_j) \subset \Lambda(z_n^j, R_j)$
is valid for every $n$.
Hence $\whw$ is essentially constant in each disc $\Delta(z_n^j, \d_j)$ by Lemma~\ref{Lemma:DD}(ii).
Moreover,
    $$|B(z)|\le \left|\frac{z_n^j-z}{1-\overline{z}_n^j z}\right|\le \d_j, \quad  z\in \Delta(z_n^j, \d_j).$$
Using these facts together with the subharmonicity of $|O_{\phi}|^p$, we obtain
    \begin{equation}\label{Eq:outer-lemma-B-est1}
    \begin{split}
    &\sum_n |O_{\phi}(z_n)|^p \frac{\whw(z_n)}{(1-|z_n|)^{p-1}}=\sum_{j=1}^{M}\sum_n |O_{\phi}(z_n^j)|^p \frac{\whw(z_n^j)}{(1-|z_n^j|)^{p-1}} \\
    &\quad\lesssim \sum_{j=1}^{M} \sum_n \int_{\Delta(z_n^j, \d_j)} |O_{\phi}(z)|^p\,dA(z) \frac{\whw(z_n^j)}{(1-|z_n^j|)^{p+1}} \\
    &\quad \asymp \sum_{j=1}^{M} \sum_n \int_{\Delta(z_n^j, \d_j)} |O_{\phi}(z)|^p \frac{\whw(z)}{(1-|z|)^{p+1}} \,dA(z) \\
    &\quad \le \sum_{j=1}^{M} (1-\d_j)^{-p}\sum_n \int_{\Delta(z_n^j, \d_j)}\left(\frac{|O_{\phi}(z)|(1-|B(z)|)}{1-|z|}\right)^p \frac{\whw(z)}{1-|z|} \,dA(z) \\
    &\quad \lesssim \int_\D \left(\frac{|O_{\phi}(z)|(1-|B(z)|)}{1-|z|}\right)^p \frac{\whw(z)}{1-|z|} \,dA(z),
    \end{split}
    \end{equation}
where $dA(z)$ is the two-dimensional Lebesgue measure.
Now it suffices to show that the last integral in \eqref{Eq:outer-lemma-B-est1} is finite.

Set $\psi(z)=\widehat{\nu}(z)/(1-|z|)$ for $z\in \D$. Note that $\widehat{\nu}(r)\asymp \widehat{\psi}(r)$ for $0\le r<1$
by Lemmas~\ref{Lemma:DD}(ii)~and~\ref{Lemma:DDD}. Moreover, integrating by parts, one can show that $\nu\in \DD_p$ if and only if
    \begin{equation}\label{Eq:DDp-hat}
    \begin{split}
    \frac{(1-r)^p}{\widehat{\nu}(r)}\int_0^r \frac{\widehat{\nu}(s)}{(1-s)^{p+1}}\,ds\asymp 1, \quad r \rightarrow 1^-.
    \end{split}
    \end{equation}
In particular, $\psi \in \RR \cap \DD_p$ by the hypotheses of $\nu$.
Since Lemma~\ref{Lemma:asymp} implies $BO_{\phi}$ in $\mathcal{B}_{\psi}^{p,p}$,
Theorem~\ref{Thm:fI-SP} gives
    $$\int_\D \left(\frac{|O_{\phi}(z)|(1-|B(z)|)}{1-|z|}\right)^p \frac{\whw(z)}{1-|z|} \,dA(z)<\infty.$$
This completes the proof.
\end{proof}

Recall that a sequence $\{z_n\}\subset \D$ is said to be uniformly separated if
    $$
    \inf_{n\in\N}\prod_{k\ne n}\left|\frac{z_k-z_n}{1-\overline{z}_kz_n}\right|>0;
    $$
and a finite union of uniformly separated sequences is called a Carleson-Newman sequence.
It is worth mentioning that any Carleson-Newman sequence is a finite union of separated sequences satisfying the Blaschke condition,
but the converse statement is not true.
For $1<p<\infty$, $p-2<\a<p-1$ and a Carleson-Newman sequence $\{z_n\}$, we can give a sufficient and necessary condition for $\{z_n\}$
to be a zero set of $\mathcal{B}_{\a}^{p,p}$.
This is a straightforward consequence of Theorem~\ref{Thm:fI-SP}, Corollary~\ref{Coro:zero} and the reasoning made in paper \cite{ABP2009} by N.~Arcozzi, D.~Blasi and J.~Pau.

\begin{corollary}\label{Coro:zero2}
Let $1<p<\infty$, $p-2<\a<p-1$ and $\{z_n\}$ be a Carleson-Newman sequence.
Then $\{z_n\}$ is a zero set of $\mathcal{B}_{\a}^{p,p}$ if and only if there exists an outer function $O_{\phi}\in \mathcal{B}_{\a}^{p,p}$ such that
    \begin{equation}\label{Eq:class-outer-cond}
    \begin{split}
    \sum_n |O_{\phi}(z_n)|^p (1-|z_n|)^{\a+2-p}<\infty.
    \end{split}
    \end{equation}
\end{corollary}

\begin{proof}
Let $B$ be the Blaschke product with zeros $\{z_n\}$ and $O_{\phi}\in \mathcal{B}_{\a}^{p,p}$ an outer function satisfying \eqref{Eq:class-outer-cond}.
Then \cite[Theorem~3.5]{Mashreghi2013} together with some elementary calculations gives
    \begin{equation}\label{Eq:ABP-est}
    \begin{split}
    &\int_\D \left(\frac{|O_{\phi}(z)|(1-|B(z)|)}{1-|z|}\right)^p (1-|z|)^{\a}\,dA(z) \\
    &\quad \le 2\int_\D |O_{\phi}(z)|^p \sum_n \frac{1-|z_n|^2}{|1-\overline{z}_nz|^2} (1-|z|)^{\a+1-p}\,dA(z) \\
    &\quad \lesssim \sum_n \int_\D |O_{\phi}(z_n)|^p \frac{1-|z_n|^2}{|1-\overline{z}_nz|^2} (1-|z|)^{\a+1-p}\,dA(z) \\
    &\quad \quad + \sum_n \int_\D |O_{\phi}(z)-O_{\phi}(z_n)|^p\frac{1-|z_n|^2}{|1-\overline{z}_nz|^2} (1-|z|)^{\a+1-p}\,dA(z) \\
    &\quad =:\mathcal{I}_1+\mathcal{I}_2.
    \end{split}
    \end{equation}
Following the reasoning in the proof of \cite[Proposition~3.2]{ABP2009}, it is easy to check that $\mathcal{I}_1$ and $\mathcal{I}_2$ are finite.
More precisely, estimating in a natural manner, one can show
    $$\mathcal{I}_1 \lesssim \sum_n |O_{\phi}(z_n)|^p (1-|z_n|)^{\a+2-p}<\infty.$$
In the argument of
    $\mathcal{I}_2\lesssim \|O_{\phi}'\|_{A_\a^{p,p}}^p<\infty$,
\cite[Lemma~2.1]{BP2008} and the hypothesis that $\{z_n\}$ is a Carleson-Newman sequence play key roles.

Since $O_{\phi}\in \mathcal{B}_{\a}^{p,p}$ and the first integral in \eqref{Eq:ABP-est} is finite, $BO_{\phi}$ belongs to $\mathcal{B}_{\a}^{p,p}$ by Theorem~\ref{Thm:fI-SP}.
Consequently, the implication $\Leftarrow$ is valid.
The converse implication is a direct consequence of Corollary~\ref{Coro:zero}. Hence the proof is complete.
\end{proof}

It is an open problem to prove a $\mathcal{B}_{\nu}^{p,p}$ counterpart of Corollary~\ref{Coro:zero2}.
One could try prove such result, for instance, assuming $\nu \in \RR \cap \DD_p$ and
    \begin{equation*}
    \sup_{0\le r<1}\frac{(1-r)^{p-1}}{\widehat{\nu}(r)}\int_r^1\frac{\nu(s)}{(1-s)^{p-1}}\,ds<\infty.
    \end{equation*}
In this case, the implication $\Leftarrow$ is the problematic part.
An idea to approach this problem is to follow the argument of \cite[Proposition~3.2]{ABP2009} and aim to apply therein
\cite[Theorem~3.1]{AC2009} instead of \cite[Lemma~2.1]{BP2008}. The down side of this method is that it leads to laborious computations of Bekoll\'e-Bonami weights.

Corollaries~\ref{Coro:zero}~and~\ref{Coro:zero2} are related to some main results in \cite{PP2011} by J.~Pau and J.~A.~ Pel\'aez.
In particular, the equivalence $\rm (i) \Leftrightarrow (ii)$ in \cite[Theorem~1]{PP2011}
follows from Corollary~\ref{Coro:zero2} by setting $p=2$.
Moreover, Corollary~\ref{Coro:zero} shows that the implication $\rm (i) \Rightarrow (ii)$ in \cite[Theorem~1]{PP2011}
is valid also if $\{z_n\}$ in the statement is a finite union of separated sequences.
Applying the last observation, we can also replace a Carleson-Newman sequence in \cite[Corollary~1]{PP2011} by a finite union of separated sequences:
If $0<\a<1$, $\{z_n\}$ is a finite union separated sequences and zero set of $\mathcal{B}_{\a}^{2,2}$, then
    $$\int_0^{2\pi} \log\left(\sum_n \frac{(1-|z_n|)^{\a+1}}{|e^{i\t}-z_n|^2}\right)\,d\t<\infty.$$
This result offers a practical way to construct Blaschke sequences which are not zero sets of $\mathcal{B}_{\a}^{2,2}$;
see \cite[Theorem~2]{PP2011} and its proof.

Note that \eqref{Eq:outer-lemma-B-est1} and \eqref{Eq:ABP-est} together with the estimates for $\mathcal{I}_1$ and $\mathcal{I}_2$
are valid also if outer function $O_\phi$ is replaced
by an arbitrary $f\in H^p$. Using this observation and Theorem~\ref{Thm:fI-SP}, we can rewrite Corollary~\ref{Coro:zero2} in the following form.

\begin{corollary}\label{Coro:B-fi-SP}
Let $1<p<\infty$, $p-2<\a<p-1$, $f\in H^p$
and $B$ be a Blaschke product associated with a Carleson-Newman sequence $\{z_n\}$.
Then $fB\in \mathcal{B}_{\a}^{p,p}$ if and only if $f\in \mathcal{B}_{\a}^{p,p}$ and
    \begin{equation*}
    \begin{split}
    \sum_n |f(z_n)|^p (1-|z_n|)^{\a+2-p}<\infty.
    \end{split}
    \end{equation*}
\end{corollary}

Corollary~\ref{Coro:B-fi-SP} is a partial improvement of the main result in M.~Jevti\'c's paper \cite{Jevtic1990}.
More precisely, this paper contains an extended counterpart of Corollary~\ref{Coro:B-fi-SP} (in the sense of $p$ and $q$) with the defect $f \equiv 1$.
It is also worth mentioning that Corollary~\ref{Coro:B-fi-SP} is not valid if the Carleson-Newman sequence $\{z_n\}$ is replaced by an arbitrary Blaschke sequence.
This can be shown by studying the case where $f\equiv 1$ and $B$ is a Blaschke product with zeros on the positive real axis.
More precisely, the counter example follows from \cite[Theorem~1]{R2016}, which asserts that all such Blaschke products
belong to $\mathcal{B}_{\a}^{p,p}$ for $1/2<p<\infty$ and $p-3/2<\a<\infty$.

Theorem~\ref{Thm:fI-SP} for $f\equiv 1$ (or Theorem~\ref{Thm3} for inner functions) has also extended counterpart \cite[Theorem~1]{R2017b}.

\begin{lettertheorem}\label{Thm:SP}
Let $0<p,q<\infty$ and $\nu \in \RR$. Then $\nu \in \DD_q$ if and only if
    \begin{equation*}
     \|I'\|_{A^{p,q}_\nu}^q \asymp \int_0^1 \left(\int_0^{2\pi} \left(\frac{1-|I(re^{i\t})|}{1-r}\right)^p d\t\right)^{q/p} \nu(r)\, dr
    \end{equation*}
for all inner functions $I$. Here the comparison constants may depend only on $p$, $q$ and $\nu$.
\end{lettertheorem}

Theorem~\ref{Thm:SP} confirms that the hypothesis $\nu \in \DD_q$ in Theorems~\ref{Thm3}~and~\ref{Thm:fI-SP} is sharp in a certain sense.
Studying the argument of this result in \cite{R2017b}, we can also deduce
that the proof of Theorem~\ref{Thm3} is more straightforward when $f$ is an inner function, and the statement is valid for all $0<p<\infty$.
It is also worth mentioning that results like Theorem~\ref{Thm:SP} have turned out to be useful in the theory of inner functions.
Several by-products of Theorem~\ref{Thm:SP} can be found in \cite{R2017b, RS2018}.

\section{Proof of Theorem~\ref{Thm1}}\label{Sec:5}

Before the proof of Theorem~\ref{Thm1} we recall \cite[Lemma~6]{R2017b},
which is a modification of \cite[Lemma~5]{Ahern1983}.

\begin{letterlemma}\label{Lemma:max-iq}
If $0<p\le 1$ and $g:[0,1) \rightarrow [0,\infty)$ is measurable, then
    $$\left(\int_r^1 g(s)ds\right)^p\le 2\int_r^1 \sup_{0\le x\le s} g(x)^p(1-s)^{p-1}ds$$
for $0\le r<1$.
\end{letterlemma}

\noindent
\emph{Proof of Theorem~\ref{Thm1}.}
Let $\frac{4}{5}\le s<1$ and choose $n=n(s)\in \N\setminus \{1,2,3,4\}$ such that $1-\frac{1}{n}\le s<1-\frac{1}{n+1}$.
Set $f_n(z)=z^n$ for $z\in \D$. Since
    \begin{equation*}
    \begin{split}
    \left|e^{in(\t+h)}-e^{in\t}\right|^2 &= \left|1-e^{inh}\right|^2= 2(1-\cos(nh))=2n^2h^2\sum_{k=1}^\infty (-1)^{k-1} \frac{(nh)^{2(k-1)}}{(2k)!} \\
    &\ge 2n^2h^2 \left(\frac{1}{2}-\frac{n^2h^2}{24}\right) \ge \frac{2}{3}n^2h^2, \quad 0<h<\frac{2}{n},
    \end{split}
    \end{equation*}
we have
    \begin{equation*}
    \begin{split}
    &\int_{1/2}^1 \om_p(1-r, f_n)^q \frac{\nu(r)}{(1-r)^q}\,dr
    \asymp \left(\int_{1/2}^{1-2/n} + \int_{1-2/n}^1\right) \sup_{0<h<1-r} |1-e^{inh}|^q \frac{\nu(r)}{(1-r)^q}\,dr \\
    &\quad \ge \int_{1/2}^{1-2/n} |1-e^{i2}|^q \frac{\nu(r)}{(1-r)^q}\,dr + \int_{1-2/n}^s |1-e^{in(1-r)}|^q \frac{\nu(r)}{(1-r)^q}\,dr \\
    &\quad \gtrsim \int_{1/2}^{1-2/n} \frac{\nu(r)}{(1-r)^q}\,dr + \int_{1-2/n}^s (n+1)^q \nu(r)\,dr \\
    &\quad \ge \int_{1/2}^{1-2/n} \frac{\nu(r)}{(1-r)^q}\,dr+(1-s)^{-q}\int_{1-2/n}^s \nu(r)\,dr \\
    &\quad \ge \int_{1/2}^s \frac{\nu(r)}{(1-r)^q}\,dr \asymp \int_0^s \frac{\nu(r)}{(1-r)^q}\,dr.
    \end{split}
    \end{equation*}
Using the hypothesis $\nu\in \DD$ together with Lemma~\ref{Lemma:DD}(iv)(ii) in a similar manner as in the proof of \cite[Theorem~1]{PR2016}, we obtain
    \begin{equation*}
    \begin{split}
    \|f_n'\|_{A_\nu^{p,q}}^q &\asymp n^q \int_0^1 r^{q(n-1)+1}\nu(r)\,dr \asymp n^q \int_{1-\frac{1}{q(n-1)+1}}^1 \nu(r)\,dr \\
    &\asymp n^q \int_{1-\frac{1}{n+1}}^1 \nu(r)\,dr \le \frac{\widehat{\nu}(s)}{(1-s)^q}.
    \end{split}
    \end{equation*}
Finally combining the estimates above and using the inequality
    $$\int_0^t \frac{\nu(r)}{(1-r)^q}\,dr \lesssim 1 \asymp \frac{\widehat{\nu}(t)}{(1-t)^q}, \quad 0<t<\frac{4}{5},$$
we deduce that if $\nu\in \DD$ and \eqref{Eq:SP} is satisfied for all $f\in H^p$, then $\nu \in \DD_p$.
Hence it suffices to prove the converse statement.

Let $f\in H^p$, $0\le \t<2\pi$, $\frac12<r<1$ and $0<h<\frac12$. Set $\r=r-h$ and
$\Gamma$ be the contour which goes first rapidly from $re^{i\t}$ to $\r e^{i\t}$,
then along the circle $\{z:|z|=\r\}$ to $\r e^{i(\t+h)}$ and finally rapidly to $re^{i(\t+h)}$. Since
    $$f(re^{i(\t+h)})-f(re^{i\t})=\int_{\Gamma} f'(z)\,dz,$$
we have
    $$|f(re^{i(\t+h)})-f(re^{i\t})|\le \int_\r^r |f'(se^{i\t})|\,ds + \int_\t^{\t+h} |f'(\r e^{it})|\,dt + \int_\r^r |f'(se^{i(\t+h)})|\,ds.$$
Consequently, the discrete and continuous forms of Minkowski's inequality, a change of variable and Hardy's convexity theorem yield
    \begin{equation}\label{Eq:thm1p-1}
    \begin{split}
    &\left(\int_0^{2\pi} |f(re^{i(\t+h)})-f(re^{i\t})|^p\,d\t\right)^{1/p}\le
    \left(\int_0^{2\pi} \left(\int_\r^r |f'(se^{i\t})|\,ds\right)^p d\t \right)^{1/p} \\
    &\quad+ \left(\int_0^{2\pi} \left(\int_0^{h} |f'(\r e^{i(x+\t)})|\,dx \right)^p d\t \right)^{1/p}
    +\left(\int_0^{2\pi} \left(\int_\r^r |f'(se^{i(\t+h)})|\,ds \right)^p d\t \right)^{1/p} \\
    &\quad \le 2 \int_\r^r M_p(s,f')\,ds + h M_p(\r,f') \le 3\int_{r-h}^r M_p(s,f')\,ds.
    \end{split}
    \end{equation}
Note that the deduction above can be found, for instance, in the proof of \cite[Theorem~5.4]{Duren1970}.

By raising both sides of \eqref{Eq:thm1p-1} to power $q$, adding $\sup_{0<h<1-t}$ and then integrating from $1/2$ to $r$ with respect to $\nu(t)\,dt/(1-t)^{q}$, we obtain
    \begin{equation*}
    \begin{split}
    &\int_{1/2}^r \sup_{0<h<1-t} \left(\int_0^{2\pi} |f(re^{i(\t+h)})-f(re^{i\t})|^p \,d\t\right)^{q/p} \frac{\nu(t)}{(1-t)^q}\,dt \\
    &\quad \lesssim \int_{1/2}^r \left(\int_{r-(1-t)}^r M_p(s,f')\,ds\right)^q \frac{\nu(t)}{(1-t)^q}\,dt.
    \end{split}
    \end{equation*}
Letting $r \rightarrow 1^-$, using the monotone and mean convergence theorems together with the hypothesis $f\in H^p$, we deduce
    \begin{equation*}
    \begin{split}
    \int_{1/2}^1 \om_p(1-t, f)^q \frac{\nu(t)}{(1-t)^q}\,dt
    \lesssim \int_{1/2}^1 \left(\int_{t}^1 M_p(s,f')\,ds\right)^q \frac{\nu(t)}{(1-t)^q}\,dt=:\mathcal{I}.
    \end{split}
    \end{equation*}
Hence it suffices to show $\mathcal{I}\lesssim \|f'\|_{A_\nu^{p,q}}^q$. Note that the argument of this estimate uses ideas from \cite{R2017b}.

If $q\le 1$, then Lemma~\ref{Lemma:max-iq} with the choice $g(s)=M_p(s,f')$, Hardy's convexity theorem, Fubini's theorem, the hypothesis $\nu\in \DD_q$ and Lemma~\ref{Lemma:asymp} give
    \begin{equation*}
    \begin{split}
    \mathcal{I}&\lesssim \int_0^1 \frac{\nu(t)}{(1-t)^q} \int_t^1 \sup_{0\le x\le s}M_p^q(x,f')(1-s)^{q-1}\,ds \,dt \\
    &=\int_0^1 \frac{\nu(t)}{(1-t)^q} \int_t^1 M_p^q(s,f')(1-s)^{q-1}\,ds \,dt \\
    &=\int_0^1 M_p^q(s,f')(1-s)^{q-1} \int_0^s \frac{\nu(t)}{(1-t)^q}\,dt\,ds \\
    &\lesssim \int_0^1 M_p^q(s,f')\frac{\widehat{\nu}(s)}{1-s}\,ds \asymp \|f'\|_{A_\nu^{p,q}}^q
    \end{split}
    \end{equation*}
for all $f \in \HO(\D)$. Hence the assertion for $q\le 1$ is proved.
If $q>1$, $0<\e<q/(\DD_q(\nu)+1)$ and $h(s)=(1-s)^{\frac{q-1-\e}{q}}$, then H\"older's inequality and Fubini's theorem yield
    \begin{equation*}
    \begin{split}
    \mathcal{I}&\lesssim \int_0^1 \int_t^1 M_p^q(s,f')h(s)^q\,ds \left(\int_t^1 h(r)^{-\frac{q}{q-1}}dr\right)^{q-1}\frac{\nu(t)}{(1-t)^q}\,dt \\
    &\asymp \int_0^1 \frac{\nu(t)}{(1-t)^{q-\e}} \int_t^1 M_p^q(s,f')(1-s)^{q-1-\e}\,ds\,dt \\
    &= \int_0^1 M_p^q(s,f')(1-s)^{q-1-\e} \int_0^s \frac{\nu(t)}{(1-t)^{q-\e}}\,dt \,ds
    \end{split}
    \end{equation*}
for all $f \in \HO(\D)$.
Since $\nu\in \DD_{q-\e}$ by Lemma~\ref{Lemma:self-imp}, the assertion for $q>1$ follows from Lemma~\ref{Lemma:asymp}. This completes the proof. \hfill$\Box$

\medskip

Since
    \begin{equation*}
    \begin{split}
    \int_0^{1/2} \om_p(1-t, f)^q \frac{\nu(t)}{(1-t)^q}\,dt \le 2^{2q} \|f\|_{H^p}^q \int_0^{1/2} \nu(t)\,dt
    \end{split}
    \end{equation*}
by Minkowski's inequality, Theorem~\ref{Thm1} has the following consequence.

\begin{corollary}\label{Coro_Thm1}
Let $1\le p<\infty$, $0<q<\infty$ and $\nu \in \RR \cap \DD_q$. Then
there exists a constant $C=C(p,q,\nu)>0$ such that
    \begin{equation*}
    \int_{0}^1 \om_p(1-r, f)^q \frac{\nu(r)}{(1-r)^q}\,dr \le C \left(\|f'\|_{A_\nu^{p,q}}^q+\|f\|_{H^p}^q \right)
    \end{equation*}
for all $f\in H^p$.
\end{corollary}

Note that Corollary~\ref{Coro_Thm1} is a part of Theorem~\ref{Thm2}.
We state it here as an independent result because it is needed for the proof of Theorem~\ref{Thm2}.

\section{Proof of Theorem~\ref{Thm2}}\label{Sec:6}

We go directly to the proof of Theorem~\ref{Thm2}.

\medskip

\noindent
\emph{Proof of Theorem~\ref{Thm2}.} Let $0\le r<1$ and $0\le t<2\pi$. Since
    $$\int_0^{2\pi} \frac{e^{i\t}d\t}{(e^{i\t}-re^{it})^2}=0,$$
Cauchy's integral formula gives
    \begin{equation*}
    \begin{split}
    |f'(re^{it})|&=\frac{1}{2\pi(1-r^2)} \left|\int_0^{2\pi}\left(f(e^{i\t})-f(re^{it})\right)\frac{e^{i\t}(1-r^2)}{(e^{i\t}-re^{it})^2}\,d\t\right| \\
    &\le \frac{1}{1-r} \int_0^{2\pi} \left|f(e^{i\t})-f(re^{it})\right| d\mu_{re^{it}}(\t), \quad f\in H^1.
    \end{split}
    \end{equation*}
Raising both sides to power $p$, integrating from $0$ to $2\pi$ with respect to $dt$,
then raising both sides to power $q/p$ and finally integrating from $0$ to 1 with respect to $\nu(r)\,dr$, we obtain
    \begin{equation} \label{Eq:thm3-1-est}
    \begin{split}
    \|f'\|_{A_\nu^{p,q}}^q
    \le \int_0^1 \left(\int_0^{2\pi} \left(\int_0^{2\pi} |f(e^{i\t})-f(re^{it})| d\mu_{re^{it}}(\t) \right)^p dt\right)^{q/p} \frac{\nu(r)}{(1-r)^q}\,dr
    \end{split}
    \end{equation}
for all $f\in H^p$.

Let $f\in H^p$, $0<q\le 1$, and set
    $$\mathcal{I}(r)=\left(\int_0^{2\pi} \left(\int_0^{2\pi} |f(e^{i\t})-f(re^{it})| d\mu_{re^{it}}(\t) \right)^p dt\right)^{q/p}.$$
By the proof of \cite[Theorem~2.1]{Dyakonov1998}, we know that
    \begin{equation}\label{Eq:thm3-dyakonov-esti}
    \begin{split}
    \mathcal{I}(r)\lesssim \left(\sum_{k=0}^\infty 2^{-k} \om_p(2^k(1-r), f) \right)^q.
    \end{split}
    \end{equation}
Hence the sub-additivity of $g(x)=x^q$ for $x\ge 0$ and Fubini's theorem give
    \begin{equation}\label{Eq:thm3-1-q<1}
    \begin{split}
    \int_0^1 \frac{\mathcal{I}(r)\nu(r)}{(1-r)^{q}}\,dr &\lesssim  \sum_{k=0}^\infty 2^{-qk} \int_0^1 \om_p(2^{k}(1-r), f)^q \frac{\nu(r)}{(1-r)^{q}}\,dr.
    \end{split}
    \end{equation}
Next we show that the weight $\nu(r)$ in the right-hand side
can be replaced by $\frac{\widehat{\nu}(r)}{1-r}$ without losing any essential information.

Set $\psi(z)=\widehat{\nu}(z)/(1-|z|)$ for $z\in \D$, and remind that $\widehat{\nu}(r)\asymp \widehat{\psi}(r)$ for $0\le r<1$
by Lemmas~\ref{Lemma:DD}(ii)~and~\ref{Lemma:DDD}. In particular, $\psi$ belongs to class $\RR$; and thus, there exist $K=K(\psi)>1$ and $C=C(\psi)>1$ such that
    \begin{equation}\label{Eq:psi-DDD}
    \begin{split}
    \widehat{\psi}(r)\ge C \widehat{\psi}\left(1-\frac{1-r}{K}\right), \quad 0\le r<1.
    \end{split}
    \end{equation}
Let $k\in \N \cup \{0\}$ and $r_n=1-K^{-n}$ for $n\in \N \cup \{0\}$.
Using \eqref{Eq:psi-DDD} together with Lemma~\ref{Lemma:DD}(ii), we obtain
    \begin{equation}\label{Eq:psi-DDD-calcu}
    \begin{split}
    (C-1)\widehat{\psi}(r_{n+1})&=C\widehat{\psi}\left(1-\frac{1-r_n}{K}\right)-\widehat{\psi}(r_{n+1})\le \widehat{\psi}(r_{n})-\widehat{\psi}(r_{n+1})\\
    &=\int_{r_n}^{r_{n+1}}\psi(r)\,dr\le \widehat{\psi}(r_{n}) \asymp \widehat{\psi}(r_{n+1}).
    \end{split}
    \end{equation}
Now Minkowski's inequality, the monotonicity of $\om_p(s, f)$ with $s$ and \eqref{Eq:psi-DDD-calcu} yield
    \begin{equation}\label{Eq:thm3-om-scal}
    \begin{split}
    \int_{0}^1  \om_p(2^k(1-r), f)^q \frac{\nu(r)}{(1-r)^{q}}\,dr
    &\lesssim \sum_{n=1}^\infty \int_{r_n}^{r_{n+1}} \om_p(2^k(1-r), f)^q \frac{\nu(r)}{(1-r)^{q}}\,dr + \|f\|_{H^p}^q \\
    &\le \sum_{n=1}^\infty \om_p(2^k(1-r_n), f)^q \frac{\widehat{\nu}(r_n)}{(1-r_{n+1})^{q}} + \|f\|_{H^p}^q \\
    &\asymp \sum_{n=1}^\infty \om_p(2^k(1-r_n), f)^q \frac{\widehat{\psi}(r_n)}{(1-r_n)^{q}} + \|f\|_{H^p}^q \\
    &\asymp \sum_{n=0}^\infty \om_p(2^k(1-r_{n+1}), f)^q \frac{\int_{r_n}^{r_{n+1}}\psi(r)\,dr}{(1-r_{n})^{q}} + \|f\|_{H^p}^q \\
    &\le \sum_{n=0}^\infty \int_{r_n}^{r_{n+1}} \om_p(2^k(1-r), f)^q \frac{\psi(r)}{(1-r)^{q}}\,dr + \|f\|_{H^p}^q \\
    &= \int_0^1  \om_p(2^k(1-r), f)^q \frac{\widehat{\nu}(r)}{(1-r)^{q+1}}\,dr + \|f\|_{H^p}^q.
    \end{split}
    \end{equation}
It is worth noting that a similar deduction works also in the opposite direction.

Using \eqref{Eq:thm3-1-q<1} and \eqref{Eq:thm3-om-scal}, we obtain
    \begin{equation*}
    \begin{split}
    \int_0^1 \frac{\mathcal{I}(r)\nu(r)}{(1-r)^{q}}\,dr &\lesssim  \left[\sum_{k=0}^\infty 2^{-qk} \int_0^{1-2^{-k}} \om_p(2^{k}(1-r), f)^q \frac{\widehat{\nu}(r)}{(1-r)^{q+1}}\,dr + \|f\|_{H^p}^q \right] \\
    &\quad + \sum_{k=0}^\infty 2^{-qk} \int_{1-2^{-k}}^1 \om_p(2^{k}(1-r), f)^q \frac{\widehat{\nu}(r)}{(1-r)^{q+1}}\,dr \\
    &=:\mathcal{I}_1+\mathcal{I}_2.
    \end{split}
    \end{equation*}
Minkowski's inequality, \eqref{Eq:DDp-hat} with $p$ being replaced by $q$ and Lemma~\ref{Lemma:DDD} yield
    \begin{equation*}
    \begin{split}
    \mathcal{I}_1&\lesssim \|f\|_{H^p}^q\sum_{k=0}^\infty 2^{-qk}\int_0^{1-2^{-k}} \frac{\widehat{\nu}(r)}{(1-r)^{q+1}}\,dr \\
    &\lesssim \|f\|_{H^p}^q\sum_{n=0}^\infty \widehat{\nu}(1-2^{-k}) \\
    &\lesssim \whw(0)\|f\|_{H^p}^q\sum_{n=0}^\infty 2^{-\a k}\asymp \|f\|_{H^p}^q
    \end{split}
    \end{equation*}
for some $\a=\a(\nu)>0$.
The continuity of $\whw$, changes of variables, Fubini's theorem and the hypothesis $\nu\in \RR$ give
     \begin{equation*}
    \begin{split}
    \mathcal{I}_2 &\asymp \int_0^1 \frac{\om_p(1-s, f)^q}{(1-s)^{q+1}}\int_0^\infty \widehat{\nu}\left(1-2^{-k}(1-s)\right)\,dk \,ds  \\
    &=\frac{1}{\log 2}\int_0^1 \frac{\om_p(1-s, f)^q}{(1-s)^{q+1}}\int_s^1 \frac{\widehat{\nu}(x)}{1-x}\,dx\,ds \\
    &\asymp \int_0^1  \om_p(1-s, f)^q \frac{\widehat{\nu}(s)}{(1-s)^{q+1}}\,ds.
    \end{split}
    \end{equation*}
Summarizing, we have shown
    \begin{equation}\label{Eq:proof-thm2-q<1}
    \begin{split}
    \int_0^1 \frac{\mathcal{I}(r)\nu(r)}{(1-r)^{q}}\,dr \lesssim \int_0^1 \om_p(1-s, f)^q \frac{\widehat{\nu}(s)}{(1-s)^{q+1}}\,ds + \|f\|_{H^p}^q.
    \end{split}
    \end{equation}
Applying a similar argument as in \eqref{Eq:thm3-om-scal}, we can replace $\widehat{\nu}(s)$ in the right-hand side of \eqref{Eq:proof-thm2-q<1} by $\nu(s)(1-s)$.
Consequently, \eqref{Eq:thm3-1-est} and Corollary~\ref{Coro_Thm1} imply \eqref{Eq:SP3} for all $f\in H^p$.
Hence the assertion for $q\le 1$ is proved.

Let $1<q<\infty$. Then \eqref{Eq:thm3-dyakonov-esti}, the continuous form of Minkowski's inequality, \eqref{Eq:thm3-om-scal} and well-known inequalities give
     \begin{equation*}
    \begin{split}
    \int_0^1 \frac{\mathcal{I}(r)\nu(r)}{(1-r)^{q}}\,dr &\lesssim  \int_0^1 \left(\sum_{k=0}^\infty 2^{-k} \om_p(2^{k}(1-r), f) \right)^q\frac{\nu(r)}{(1-r)^{q}}\,dr \\
    &\le \left(\sum_{k=0}^\infty 2^{-k}\left(\int_{0}^1  \om_p(2^{k}(1-r), f)^q\frac{\nu(r)}{(1-r)^{q}} \,dr\right)^{1/q}\right)^{q} \\
    &\lesssim \left(\sum_{k=0}^\infty 2^{-k}\left(\int_{0}^1  \om_p(2^{k}(1-r), f)^q\frac{\widehat{\nu}(r)}{(1-r)^{q+1}} \,dr\right)^{1/q}\right)^{q} + \|f\|_{H^p}^q  \\
    &\asymp \left[\left(\sum_{k=0}^\infty 2^{-k}\left(\int_0^{1-2^{-k}}  \om_p(2^{k}(1-r), f)^q\frac{\widehat{\nu}(r)}{(1-r)^{q+1}} \,dr\right)^{1/q}\right)^{q} + \|f\|_{H^p}^q \right]  \\
    &\quad +\left(\sum_{k=0}^\infty 2^{-k}\left(\int_{1-2^{-k}}^1  \om_p(2^{k}(1-r), f)^q\frac{\widehat{\nu}(r)}{(1-r)^{q+1}} \,dr\right)^{1/q}\right)^{q} \\
    &=:\mathcal{I}_3 + \mathcal{I}_4.
    \end{split}
    \end{equation*}
Minkowski's inequality, \eqref{Eq:DDp-hat} with $p$ being replaced by $q$ and Lemma~\ref{Lemma:DDD} yield
    \begin{equation*}
    \begin{split}
    \mathcal{I}_3 &\lesssim \|f\|_{H^p}^q\left(\sum_{k=0}^\infty 2^{-k}\left(\int_0^{1-2^{-k}} \frac{\widehat{\nu}(r)}{(1-r)^{q+1}} \,dr\right)^{1/q}\right)^{q} \\
    &\lesssim \|f\|_{H^p}^q\left(\sum_{k=0}^\infty \whw(1-2^{-k})^{1/q}\right)^{q} \asymp \|f\|_{H^p}^q.
    \end{split}
    \end{equation*}
By Lemma~\ref{Lemma:DDD}, there exists a constant $\a=\a(\nu)>0$ such that
    $$\widehat{\nu}\left(1-2^{-k}(1-s)\right)\lesssim 2^{-\a k}\widehat{\nu}(s), \quad 0\le s\le 1-2^{-k}(1-s)<1.$$
Using this together with a change of variable and modification of \eqref{Eq:thm3-om-scal}, we get
    \begin{equation}\label{Eq:thm3-q>1-last}
    \begin{split}
    \mathcal{I}_4&= \left(\sum_{k=0}^\infty \left(\int_0^1  \om_p(2^{k}(1-s), f)^q\frac{\widehat{\nu}\left(1-2^{-k}(1-s)\right)}{(1-s)^{q+1}} \,ds\right)^{1/q}\right)^{q} \\
    &\lesssim \left(\sum_{k=0}^\infty 2^{-\a k/q}\left(\int_0^1 \om_p(1-s, f)^q\frac{\widehat{\nu}(s)}{(1-s)^{q+1}} \,ds\right)^{1/q}\right)^{q} \\
    &\asymp \int_0^1\om_p(1-s, f)^q \frac{\widehat{\nu}(s)}{(1-s)^{q+1}}\,ds \\
    &\lesssim \int_0^1\om_p(1-s, f)^q \frac{\nu(s)}{(1-s)^{q}}\,ds + \|f\|_{H^p}^q.
    \end{split}
    \end{equation}
Finally \eqref{Eq:thm3-1-est}, \eqref{Eq:thm3-q>1-last} and Corollary~\ref{Coro_Thm1} imply \eqref{Eq:SP3} for all $f\in H^p$.
This completes the proof. \hfill$\Box$

\section{Proof of Theorem~\ref{Thm3}}\label{Sec:7}

Before the proof of Theorem~\ref{Thm3} we recall the following result, which is a part of the argument of \cite[Theorem~1.1]{Boe2003}.

\begin{letterlemma}\label{lemma:outer-upper}
If $O_{\phi}$ is an outer function, then
    \begin{equation*}
    \begin{split}
    |O_{\phi}'(z)|\le \frac{4}{1-|z|} \left(\int_0^{2\pi} \left|\phi(e^{i\t})-\int_0^{2\pi} \phi(e^{is}) d\mu_z(s)\right| d\mu_z(\t)
    +\int_0^{2\pi} \phi(e^{ih}) d\mu_z(h) - |O_{\phi}(z)| \right)
    \end{split}
    \end{equation*}
for all $z\in \D$.
\end{letterlemma}

\medskip

\noindent
\emph{Proof of Theorem~\ref{Thm3}.} Let $f \in H^p$.
Then there exist an inner function $I$ and an outer function $O_{\phi}$ such that $f=IO_{\phi}$. Hence the Schwarz-Pick lemma,
Lemma~\ref{lemma:outer-upper} and the fact that $\phi(\xi)=|f(\xi)|$ for almost every $\xi \in \T$ yield
    \begin{equation}\label{Eq:thm3-slow-est}
    \begin{split}
    &|f'(z)|(1-|z|)\le \left(|I(z)O_{\phi}'(z)| + |I'(z)O_{\phi}(z)|\right)(1-|z|) \\
    &\quad \le |O_{\phi}'(z)|(1-|z|) + 2|O_{\phi}(z)|(1-|I(z)|)\\
    &\quad \le 4\int_0^{2\pi} \left||f(e^{i\t})|-\int_0^{2\pi} |f(e^{is})| d\mu_z(s)\right| d\mu_z(\t)
    +4\left(\int_0^{2\pi} |f(e^{ih})| d\mu_z(h) - |f(z)|\right)
    \end{split}
    \end{equation}
for all $z\in \D$. Write $z=re^{it}$. Raising both sides of \eqref{Eq:thm3-slow-est} to power $p$, integrating from $0$ to $2\pi$ with respect to $dt$, then raising both sides to power $q/p$, integrating from $0$ to $1$ with respect to $\nu(r) dr/(1-r)^q$ and finally splitting the right-hand side into two parts, we obtain
    $$\|f'\|_{A_\nu^{p,q}}^q \lesssim F_1(f)+F_2(f),$$
which is the first inequality in \eqref{Eq:SP4}.

Set
    \begin{equation*}
    \begin{split}
    \Gamma=\Gamma(z,f)=\left\{\t\in [0,2\pi):\int_0^{2\pi} |f(e^{is})| d\mu_z(s)\le |f(e^{i\t})|\right\}, \quad z\in \D.
    \end{split}
    \end{equation*}
Then elementary calculations together with the subharmonicity of $|f|$ yield
    \begin{equation*}
    \begin{split}
    &\int_0^{2\pi} \left||f(e^{i\t})|-\int_0^{2\pi} |f(e^{is})| d\mu_z(s)\right| d\mu_z(\t) \\
    &\quad= 2\int_{\Gamma}\left(|f(e^{i\t})|-\int_0^{2\pi} |f(e^{is})| d\mu_z(s)\right)\,d\mu_z(\t) \\
    &\quad \le 2 \int_{\Gamma} \left(|f(e^{i\t})|-|f(z)|\right) d\mu_{z}(\t), \quad z\in \D.
    \end{split}
    \end{equation*}
It follows that
    \begin{equation*}
    \begin{split}
    &\int_0^{2\pi} \left||f(e^{i\t})|-\int_0^{2\pi} |f(e^{is})| d\mu_z(s)\right| d\mu_z(\t)
    +\left(\int_0^{2\pi} |f(e^{ih})| d\mu_z(h) - |f(z)|\right) \\
    &\quad \le 2 \left(\int_{\Gamma}+\int_0^{2\pi}\right) \left|f(e^{i\t})-f(z)\right| d\mu_{z}(\t) \\
    &\quad \le 4 \int_0^{2\pi} \left|f(e^{i\t})-f(z)\right| d\mu_{z}(\t), \quad z\in \D.
    \end{split}
    \end{equation*}
Doing a corresponding integration procedure for this estimate as above and applying Theorem~\ref{Thm2}, we obtain
    $$F_1(f)+F_2(f)\lesssim \|f'\|_{A_\nu^{p,q}}^q + \|f\|_{H^p}^q,$$
which is the last inequality in \eqref{Eq:SP4}. This completes the proof. \hfill$\Box$

\medskip

\noindent
\textbf{Acknowledgements.} The author thanks Antti Per\"al\"a and Toshiyuki Sugawa for valuable comments,
Tohoku University for hospitality during his visit there, and the referees for careful reading of the manuscript.

\end{document}